\documentclass[12pt, reqno]{amsart}
\usepackage{amssymb,latexsym,amsmath,amscd,amsthm,graphicx, color}
\usepackage[all]{xy}
\usepackage{pgf,tikz}
\usepackage{mathrsfs}
\usetikzlibrary{arrows}
\usepackage[left=0.6 in, top=0.6 in, right=0.6 in, bottom=0.3 in]{geometry}
\makeatletter
\newcommand{\setword}[2]{%
  \phantomsection
  #1\def\@currentlabel{\unexpanded{#1}}\label{#2}%
}
\makeatother

\usepackage[linkcolor=blue, urlcolor=blue, citecolor=blue,
colorlinks, bookmarks]{hyperref}

\definecolor{uuuuuu}{rgb}{0.26666666666666666,0.26666666666666666,0.26666666666666666}
\definecolor{xdxdff}{rgb}{0.49019607843137253,0.49019607843137253,1.}
\definecolor{ffqqqq}{rgb}{1.,0.,0.}
\definecolor{ffqqqq}{rgb}{1.,0.,0.}
\definecolor{ffxfqq}{rgb}{1.,0.4980392156862745,0.}

\definecolor{uuuuuu}{rgb}{0.26666666666666666,0.26666666666666666,0.26666666666666666}
\definecolor{qqwuqq}{rgb}{0.,0.39215686274509803,0.}
\definecolor{zzttqq}{rgb}{0.6,0.2,0.}
\definecolor{xdxdff}{rgb}{0.49019607843137253,0.49019607843137253,1.}
\definecolor{qqqqff}{rgb}{0.,0.,1.}
\definecolor{cqcqcq}{rgb}{0.7529411764705882,0.7529411764705882,0.7529411764705882}
\definecolor{sqsqsq}{rgb}{0.12549019607843137,0.12549019607843137,0.12549019607843137}

\theoremstyle{plain}

\newtheorem{theorem}[subsection]{Theorem}

\newtheorem{corollary}[subsection]{Corollary}
\newtheorem{lemma}[subsection]{Lemma}

\newtheorem{prop}[subsection]{Proposition}

\theoremstyle{definition}

\newtheorem{exam}[subsection]{Example}

\newtheorem{remark}[subsection]{Remark}

\newtheorem{note}[subsection]{Note}

\newcommand{\uu}{\cup}
\newcommand{\ii}{\cap}
\newcommand{\UU}{\bigcup}

\newcommand{\ci}{\subseteq}
\newcommand{\sci}{\subset}

\newcommand{\set}[1]{\{#1\}}

\newcommand{\ga}{\alpha}

\newcommand{\gk}{\kappa}

\newcommand{\tit}{\textit}

\newcommand{\D}[1]{\mathbb{#1}}

\newcommand{\te}{\text}

\newcommand{\pa}{\partial}

\newcommand{\tri}{\triangle}
\begin{document}

To appear, Houston Journal of Mathematics. 
\title{Constrained quantization for a uniform distribution}

\address{School of Mathematical and Statistical Sciences\\
The University of Texas Rio Grande Valley\\
1201 West University Drive\\
Edinburg, TX 78539-2999, USA.}

\email{\{$^1$pigar.biteng01, $^2$mathieu.caguiat01, $^3$dipok.deb01, $^4$mrinal.roychowdhury\}@utrgv.edu}
\email{$^5$betyvelav@gmail.com}
\author{$^1$Pigar Biteng}
\author{$^2$Mathieu Caguiat}
\author{$^3$Dipok Deb}
\author{$^4$Mrinal Kanti Roychowdhury}
 \author{$^5$Beatriz Vela Villanueva}



\date{}
\maketitle

\pagestyle{myheadings}\markboth{P. Biteng, M. Caguiat, D. Deb,  M.K. Roychowdhury, and B. Villanueva}{Constrained quantization for a uniform distribution}

\begin{abstract}
Constrained quantization for a Borel probability measure refers to the idea of estimating a given probability by a discrete probability with a finite number of supporting points lying on a specific set. The specific set is known as the constraint of the constrained quantization. A quantization without a constraint is known as an unconstrained quantization, which traditionally in the literature is known as quantization. Constrained quantization has recently been introduced by Pandey and Roychowdhury. In this paper, for a uniform distribution with support lying on a side of an equilateral triangle, and the constraint as the union of the other two sides, we obtain the optimal sets of $n$-points and the $n$th constrained quantization errors for all positive integers $n$. We also calculate the constrained quantization dimension and the constrained quantization coefficient. 
\end{abstract}

\section{Introduction}
Let $P$ be a Borel probability measure on $\D R^2$ equipped with a Euclidean metric $d$ induced by the Euclidean norm $\|\cdot\|$. For a given closed subset $S$ of $\D R^2$ the distortion error for $P$ with respect to a set $\ga \ci S$, denoted by $V(P; \ga)$, is defined as
\begin{equation*} \label{Vr}  V(P; \ga)= \int \mathop{\min}\limits_{a\in\ga} d(x, a)^2 dP(x).
\end{equation*}
Then, for $n\in \mathbb{N}$, the \tit {$n$th constrained quantization
error} for $P$ is defined as
\begin{equation} \label{Vr} V_{n}:=V_{n}(P)=\inf \Big\{V(P; \ga) : \ga \ci S, 1\leq  \text{card}(\ga) \leq n \Big\},\end{equation}
where $\te{card}(A)$ represents the cardinality of a set $A$. The set $S$ is known as the \tit{constraint} of the constrained quantization error. If $S=R^2$, i.e., if there is no constraint, then the definition of constrained quantization error reduces to the definition of unconstrained quantization error, which traditionally in the literature is known as quantization.    
For more details about unconstrained quantization error, one can see \cite{DFG, DR, GG, GL, GL1, GL2, GL3, GN, KNZ, P, P1, R1, R2, R3, Z1, Z2}. On the other hand, constrained quantization has recently been introduced in \cite{PR1, PR2} and it has much more broad applications in the areas such as information theory, machine learning and data compression, signal processing and national security.
We assume that $\int d(x, 0)^r dP(x)<\infty$ to make sure that the infimum in \eqref{Vr} exists (see \cite{PR1}).  A set $ \ga \ci S$ for which the infimum in  \eqref{Vr} exists and does not contain more than $n$ elements is called an \tit{optimal set of $n$-points} for $P$. Elements of an optimal set are called \tit{optimal elements}. 
 Let $V_{n}(P)$ be a strictly decreasing sequence, and write $V_{\infty}(P):=\mathop{\lim}\limits_{n\to \infty} V_{n}(P)$. 
Then, the number $D(P)$ defined by  
\begin{align*}
D(P):=\mathop{\lim}\limits_{n\to \infty}  \frac{2\log n}{-\log (V_{n}(P)-V_{\infty}(P))}
\end{align*} 
if it exists, 
 is called the \tit{constrained quantization dimension} of $P$ and is denoted by $D(P)$. The constrained quantization dimension measures the speed at which the specified measure of the constrained quantization error converges as $n$ tends to infinity.
For any $\gk>0$, the number 
\[\lim_n n^{\frac 2 \gk}  (V_{n}(P)-V_{\infty}(P)),\]
if it exists, is called the \tit{$\gk$-dimensional constrained quantization coefficient} for $P$.

Let $\tri OAB$ be an equilateral triangle with vertices $O(0,0)$, $A(2, 0)$, $B(1,\sqrt 3)$. Let $P$ be a Borel probability measure on $\D R$ which has support the closed interval $[0, 2]$. Moreover, $P$ is uniform on its support. Let $S_1$ be the side $OB$ and $S_2$ be the side $AB$. In this paper, we take the constraint as $S:=S_1\uu S_2$. Then, in Section~\ref{sec3} for the uniform distribution $P$ we obtain the optimal sets of $n$-points and the $n$th constrained quantization errors for all positive integers $n$. In Section~\ref{sec4}, we calculate the constrained quantization dimension and the constrained quantization coefficient. As mentioned in Remark~\ref{rem1}, though the constrained quantization dimension of the uniform distribution defined in this paper is one, which is equal to the Euclidean dimension of the space where the support of the probability measure is defined, it is not always true (see \cite{PR1, PR3}). 

\begin{remark}
In this paper, to keep our calculation simple to investigate the constrained quantization, we have considered the equilateral triangle with height $\sqrt 3$, and the probability measure as the uniform distribution on the base of the triangle. People can consider any triangle with any height and any probability measure with support on any side of the triangle. 
\end{remark}

\section{Preliminaries}
With respect to a finite set $\ga \sci \D R^2$, by the \tit{Voronoi region} of an element $a\in \ga$, it is meant the set of all elements in $\D R^2$ which are nearest to $a$ among all the elements in $\ga$, and is denoted by $M(a|\ga)$.  
For any two elements $(a, b)$ and $(c, d)$ in $\D R^2$, we write 
 \[\rho((a, b), (c, d)):=(a-c)^2 +(b-d)^2,\] which gives the squared Euclidean distance between the two elements $(a, b)$ and $(c, d)$.
  Two elements $p$ and $q$ in an optimal set of $n$-points are called \tit{adjacent elements} if they have a common boundary in their own Voronoi regions. Let $e$ be an element on the common boundary of the Voronoi regions of two adjacent elements $p$ and $q$ in an optimal set of $n$-points. Since the common boundary of the Voronoi regions of any two elements is the perpendicular bisector of the line segment joining the elements, we have
\[\rho(p, e)-\rho(q, e)=0. \]
We call such an equation a \tit{canonical equation}. 
Notice that any element $x\in \D R$ can be identified as an element $(x, 0)\in \D R^2$. Thus, the nonnegative real-valued function $\rho$ on $\D R \times \D R^2$ defined by 
\[\rho: \D R \times \D R^2 \to [0, \infty) \te{ such that } \rho(x, (a, b))=(x-a)^2 +b^2,\]
represents the squared Euclidean distance between an element $x\in \D R$ and an element $(a, b)\in \D R^2$. 
 Recall that $P$ is a Borel probability measure on $\D R$ which is uniform on its support the closed interval $[0, 2]$. Hence, the probability density function $f$ for $P$ is given by 
\[f(x)=\left\{\begin{array}{cc}
 \frac 1 2 & \te{ if } 0\leq x\leq 2,\\
 0 & \te{ otherwise}.
\end{array}\right.
\]
Hence, we have $dP(x)=P(dx)=f(x) dx$ for any $x\in \D R$. 
Let the constraint $S$ be the union of the sides of the equilateral triangle except the base, i.e., $S=S_1\uu S_2$, where
\begin{align*}
S_1:&=\set{(x, y)\in \D R^2 : 0\leq x\leq 1 \te{ and } y=\sqrt 3 x} \te{ and } \\
S_2:&=\set{(x, y)\in \D R^2 : 1\leq x\leq 2 \te{ and } y=-\sqrt 3(x-2)}.
\end{align*}

\begin{note} \label{note11}
The perpendicular through a point $(a, a\sqrt 3)$ on the line $y= \sqrt 3x$ intersects the support of $P$ at the point $(4a, 0)$ so that $0\leq 4a\leq 2$, i.e., $0\leq a\leq \frac 12$. The perpendicular through a point $(b, -\sqrt 3(b-2))$ on the line $y=-\sqrt 3(x-2)$ intersects the support of $P$ at the point $(4b-6,0)$ so that $0\leq 4b-6\leq 2$, i.e., $\frac 32\leq b\leq 2$. 
Let $\ga_n$ be an optimal set of $n$-points for some $n\in \D N$. Since $S:=S_1\uu S_2$ is the constraint, we must have 
\[\ga_n=(\ga_n\ii S_1)\uu (\ga_n\ii S_2).\]
For some $\ell, m\in \D N$, let 
$\ga_n\ii S_1=\set{(a_1, a_1\sqrt 3), (a_2, a_2\sqrt 3), \cdots, (a_\ell, a_\ell\sqrt 3)}$, and $\ga_n\ii S_2=\set{(b_m, -(b_m-2)\sqrt 3), (b_{m-1}, -(b_{m-1}-2)\sqrt 3), \cdots, (b_1, -(b_1-2)\sqrt 3)}$. Then, without any loss of generality we can assume that
\begin{equation} \label{D1} 0\leq a_1<a_2<\cdots<a_\ell\leq \frac 12<\frac 32\leq b_m<b_{m-1}<\cdots<b_2<b_1\leq 2.\end{equation} 
Under the assumption \eqref{D1}, the elements $(a_i,a_i\sqrt 3)$ and $(a_{i+1}, a_{i+1}\sqrt 3)$ for $1\leq i\leq \ell-1$;  $(a_\ell,a_\ell\sqrt 3)$ and $(b_m, -(b_m-2)\sqrt 3)$; and $(b_{j+1}, -(b_{j+1}-2)\sqrt 3)$ and $(b_j, -(b_j-2)\sqrt 3)$ for $1\leq j\leq m-1$, are the adjacent elements.  
 \end{note}
 
The following proposition that is well-known in unconstrained quantization is also true in constrained quantization.  
\begin{prop}  \label{prop000} 
Let $P$ be a continuous Borel probability measure. Let $\ga_n$ be an optimal set of $n$-points for $n\in \D N$. Then, 
$(i)$ $P(M(a|\ga_n))>0$, and $(ii)$ $P(\pa M(a|\ga_n))=0$ for $a\in \ga_n$, where $\pa M(a|\ga_n)$ represents the boundary of the Voronoi region $M(a|\ga_n)$.  
\end{prop}

\begin{remark}  
Let $\ga_n$ be an optimal set of $n$-points for some $n\in \D N$. Let $\ga_n\ii S_1$ and $\ga_n\ii S_2$ be the sets as defined in Note~\ref{note11}. Then, there is no common boundary between the Voronoi regions of any two elements in $\ga_n$ unless the two elements are adjacent elements.   
\end{remark}

In the following sections, we give the main results of the paper.

\section{Optimal sets of $n$-points for all $n\geq 1$} \label{sec3} 
In this section, we determine the optimal sets of $n$-points and the $n$th constrained quantization errors for all $n\in \D N$ under the condition that the optimal elements lie on the set $S$.

\begin{prop}
An optimal set of one-point is given by $\set{(\frac 14, \frac{\sqrt{3}}{4})}$ or $\set{(\frac{7}{4},\frac{\sqrt{3}}{4})}$ with constrained quantization error $\frac {13}{12}$.
\end{prop}

\begin{proof} Let $\ga:=\set{(a, b)}$ be an optimal set of one-point. The following two cases can happen:

\tit{Case~1. $(a, b)\in S_1$.}

In this case, as $b= a \sqrt 3$, the distortion error is given by
\[V(P; \ga)=\int \rho(x, (a, a\sqrt 3)) \, dP(x)=\frac 12 \int_0^2\rho(x, (a, a\sqrt 3)) \,dx=4 a^2-2 a+\frac{4}{3},\]
the minimum value of which is $\frac{13}{12}$, and it occurs when $a=\frac 14$.

\tit{Case~2. $(a, b)\in S_2$.}

In this case, $b=-(a-2)\sqrt 3$. Then, the distortion error is given by
\[V(P; \ga)=\int \rho(x, (a, -(a-2)\sqrt 3)) \, dP(x)=\frac 12 \int_0^2\rho(x, (a,  (a, -(a-2)\sqrt 3))) \,dx=4 a^2-14 a+\frac{40}{3},\]
the minimum value of which is $\frac{13}{12}$, and it occurs when $a=\frac 74$.

Hence, an optimal set of one-point is given by either $\set{(\frac 14, \frac{\sqrt{3}}{4})}$ or $\set{(\frac{7}{4},\frac{\sqrt{3}}{4})}$ with constrained quantization error $\frac {13}{12}$, which is the proposition (see Figure~\ref{Fig1}).
\end{proof}

\begin{prop} \label{propD1} 
The optimal set of two-points is $\set{(\frac 18, \frac 18\sqrt 3), (\frac {15}8, \frac 18\sqrt 3)}$ with constrained quantization error $V_2=\frac{13}{48}$
\end{prop}

\begin{proof}
Let $\ga$ be an optimal set of two-points. The following cases can arise:

\tit{Case~1: Both the elements are on $S_1$.}

In this case, we can assume that the two elements are $(a_1, a_1\sqrt 3)$ and $(a_2, a_2\sqrt 3)$. Since the probability measures of the Voronoi regions are positive, we have $0\leq a_1<a_2\leq \frac 12$.  The perpendicular bisector of the line segment joining the two elements intersects the support of $P$ at the point
$(2(a_1+a_2),0)$. Hence, the distortion error is given by
\begin{align*}
V(P; \ga)&=\int\min_{a\in \ga} \rho(x, a) \,dP(x)=\int_0^{2(a_1+a_2)}\min_{a\in \ga} \rho(x, a) \,dP(x)+\int_{2(a_1+a_2)}^2\min_{a\in \ga} \rho(x, a) \,dP(x)\\
&=\int_0^{2(a_1+a_2)}\rho(x, (a_1, a_1\sqrt 3)) f(x)\,dx+\int_{2(a_1+a_2)}^2\rho((x, 0), (a_2, a_2\sqrt 3)) f(x)\,dx\\
&=2 a_1^3+2 a_2 a_1^2-2 a_2^2 a_1-2 \left(a_2-1\right){}^2 a_2+\frac{4}{3},
\end{align*}
the minimum value of which is $\frac{49}{48}$, and it occurs when $a_1=\frac 18$ and $a_2=\frac 38$.

\tit{Case~2: Both the elements are on $S_2$.}

This case is the reflection of Case~1 with respect to the line $x=1$. 
Hence, if the two elements are $(b_1, -(b_1-2)\sqrt 3)$ and $(b_2, -(b_2-2)\sqrt 3)$, then
the distortion error is $V(P; \ga)=\frac{49}{48}$, and it occurs when $b_1=\frac {15}{8}$ and $b_2=\frac{13}8$.

\tit{Case~3: One element is on $S_1$ and one element is on $S_2$.}

In this case, let the two elements be $(a_1, a_1\sqrt 3)$ and $(b_1, -(b_1-2)\sqrt 3)$. Let the boundary of the Voronoi regions of the two elements $(a_1, a_1\sqrt 3)$ and $(b_1, -(b_1-2)\sqrt 3)$ intersect the support of $P$ at the point $(d,0)$. Then, the canonical equation is 
\[\rho((a_1, a_1\sqrt 3), (d, 0))-\rho((b_1, -(b_1-2)\sqrt 3), (d, 0))=0\]
implying $d=\frac{2 \left(a_1^2-b_1^2+3 b_1-3\right)}{a_1-b_1}$. 
Hence, the distortion error $V(P; \ga)$ is given by
\begin{align*}
V(P; \ga)&=\int\min_{a\in \ga} \rho(x, a) \,dP(x)\\
&=\int_0^{d}  \rho(x, (a_1, a_1\sqrt 3)) \,dP(x)+\int_{d}^2 \rho(x, (b_1, -(b_1-2)\sqrt 3)) \,dP(x)\\
&=\frac{2 \left(-6 a_1^2 \left(b_1^2-3 b_1+3\right)+a_1 \left(6 b_1^2-21 b_1+20\right)+3 a_1^4+3 b_1^4-24 b_1^3+66 b_1^2-74 b_1+27\right)}{3 \left(a_1-b_1\right)},
\end{align*}
the minimum value of which is $\frac{13}{48}$, and it occurs when $a_1=\frac 18$ and $b_1=\frac {15}8$.

Among all the above possible cases, we see that the distortion error is smallest in Case~3. Hence, the optimal set of two-points is $\set{(\frac 18, \frac 18\sqrt 3), (\frac {15}8, \frac 18\sqrt 3)}$ with constrained quantization error $V_2=\frac{13}{48}$ (see Figure~\ref{Fig1}). Thus, the proof of the proposition is complete.
\end{proof}

\begin{prop} \label{propD2}
The set $\ga:=\set{(a_1, a_1\sqrt 3), (a_2, a_2\sqrt 3), (b_1, -(b_1-2)\sqrt 3)}$, where 
\[a_1=\frac{1}{12} \left(26-7 \sqrt{13}\right), a_2=\frac{1}{4} \left(26-7 \sqrt{13}\right), \te{ and } b_1=\frac{73}{12}-\frac{7 \sqrt{13}}{6},\] forms an optimal set of three-points with quantization error $V_3=\frac{1}{108} (-637) \left(28 \sqrt{13}-101\right)$. 
\end{prop}

\begin{proof}
Let $\ga$ be an optimal set of three-points. The following cases can happen: 

\tit{Case~1: All the three elements lie on one of the sides $S_1$ and $S_2$. }

Without any loss of generality, we can assume that all the elements lie on the side $S_1$. Then,
\[\ga=\set{(a_1, a_1\sqrt 3 ), (a_2,   a_2\sqrt 3 ), (a_3,   a_3\sqrt 3 )}\]
such that by Note~\ref{note11}, we have  $0\leq a_1<a_2<a_3\leq \frac 12$. Notice that the boundary line of the Voronoi regions of the elements $(a_1, \sqrt 3 a_1)$ and $(a_2, \sqrt 3 a_2)$ intersects the support of $P$ at the point $(2(a_1+a_2), 0)$; the boundary line of the Voronoi regions of the elements $(a_2, a_2\sqrt 3 )$ and $(a_3, a_3\sqrt 3 )$ intersects the support of $P$ at the point $(2(a_2+a_3), 0)$. Hence, if $V(P; \ga)$ is the distortion error, then
\begin{align*}
V(P; \ga)&=\int\min_{a\in \ga}\rho(x, a)\,dP=\int_0^{2(a_1+a_2)} \rho(x, (a_1, a_1\sqrt 3))\,dP+\int_{2(a_1+a_2)}^{2(a_2+a_3)} \rho(x, (a_2, a_2\sqrt 3))\,dP\\
&\qquad \qquad +\int_{2(a_2+a_3)}^{2} \rho(x, (a_3, a_3\sqrt 3))\,dP\\
&=2 a_1^3+2 a_2 a_1^2-2 a_2^2 a_1-2 a_3^3-2 \left(a_2-2\right) a_3^2+2 \left(a_2^2-1\right) a_3+\frac{4}{3},
\end{align*}
the minimum value of which is $\frac{109}{108}$ and it occurs when $a_1=\frac{1}{12}, \, a_2= \frac{1}{4}, \, a_3= \frac{5}{12}$.

\tit{Case~2: Two elements are on one side and one element is on another side. }

Without any loss of generality, we can assume that two elements lie on $S_1$ and one element lies on $S_2$. Then, 
\[\ga=\set{(a_1,  a_1\sqrt 3), (a_2,  a_2\sqrt 3 ), (b_1,  -(b_1-2)\sqrt 3)},\]
 where $0<a_1<a_2<1<b_1<2$. 
Let the boundaries of the Voronoi regions of $(a_1, a_1\sqrt 3 )$ and $(a_2,  a_2\sqrt 3)$
intersect the support of $P$ at the point $(2(a_1+a_2), 0)$. Let the boundary of the Voronoi regions of $(a_2,  a_2\sqrt 3)$ and $(b_1,  -(b_1-2)\sqrt 3)$ intersect the support of $P$ at the points $(d, 0)$. Then, solving the canonical equation
\[\rho((a_2, a_2\sqrt 3), (d, 0))- \rho((b_1, -(b_1-2)\sqrt 3), (d, 0))=0,\]
we have 
\[d=\frac{2 \left(a_2^2-b_1^2+3 b_1-3\right)}{a_2-b_1}.\]
Then, the distortion error is given by
\begin{align*}
V(P; \ga)&=\int\min_{a\in \ga}\rho(x, a)\,dP\\
&=\int_0^{2(a_1+a_2)} \rho(x, (a_1, \sqrt 3a_1))\,dP+\int_{2(a_1+a_2)}^{d} \rho(x, (a_2, \sqrt 3a_2))\,dP\\
&\qquad \qquad +\int_{d}^{2} \rho(x, (b_1, -(b_1-2)\sqrt 3))\,dP.
\end{align*}
The distortion error is minimum if 
\[\frac{\pa}{\pa a_i}(V(P; \ga))=0 \te{ for  } i=1, 2, \te{ and } \frac{\pa}{\pa b_1}(V(P; \ga))=0.\]
Upon simplification, we see that $V(P; \ga)$ is minimum if $a_1=\frac{1}{12} \left(26-7 \sqrt{13}\right)$, $a_2=\frac{1}{4} \left(26-7 \sqrt{13}\right)$, and $b_1=\frac{73}{12}-\frac{7 \sqrt{13}}{6}$, and then $V(P; \ga)=\frac{1}{108} (-637) \left(28 \sqrt{13}-101\right)$. 

Considering Case~1 and Case~2, we see that the distortion error is minimum in Case~2, i.e., the set $\ga:=\set{(a_1, a_1\sqrt 3), (a_2, a_2\sqrt 3), (b_1, -(b_1-2)\sqrt 3)}$ in Case~2 forms an optimal set of three-points with quantization error $V_3=\frac{1}{108} (-637) \left(28 \sqrt{13}-101\right)$, which is the proposition (see Figure~\ref{Fig1}). 
\end{proof}

\begin{remark} \label{remD1} 
Proceeding in the similar way as Proposition~\ref{propD2}, we can show that the optimal set of four-points is given by
\[\ga_4=\Big\{(a_i, a_i\sqrt 3) : 1\leq i\leq 2\Big\}\UU \Big\{(b_j, -(b_j-2)\sqrt 3) :  1\leq j\leq 2\Big\},\]
where $a_1=\frac 1{16}, \, a_2=\frac{3}{16}, \, b_1=\frac{31}{16}$, and $b_2=\frac {29}{16}$ with constrained quantization error $V_4=\frac{49}{192}$. An optimal set of five-points is given by \[\ga_5=\Big\{(a_i, a_i\sqrt 3) : 1\leq i\leq 3\Big\}\UU \Big\{(b_j, -(b_j-2)\sqrt 3) :  1\leq j\leq 2\Big\},\]
where $a_i=\frac {(2i-1)u}{2}$ and $b_j=2-\frac {(2j-1)v}{2}$ for $1=1, 2, 3$ and $j=1, 2$, with 
\[u=\frac{7}{10} \left(21-2 \sqrt{109}\right) \te{ and } v=\frac{1}{10} \left(21 \sqrt{109}-218\right).\] The corresponding constrained quantization error is given by 
\[V_5=\frac{1}{300} (-5341) \left(84 \sqrt{109}-877\right).\]
To know the pictorial representation, one can see Figure~\ref{Fig1}.
\end{remark}

\begin{lemma}\label{lemmaD1}
For $n\geq 2$, let $\ga_n$ be an optimal set of $n$-points for $P$. Then, $\ga_n$ contains points from both $S_1$ and $S_2$. 
\end{lemma} 
\begin{proof}
By Proposition~\ref{propD1}, the lemma is true for $n=2$. Let us now prove the lemma for $n\geq 3$. Let $\ga_n$ be an optimal set of $n$-points for $n\geq 3$ with the corresponding $n$th constrained quantization error $V_n$. For the sake of contradiction, without any loss of generality, we can assume that all the elements in $\ga_n$ lie on the side $S_1$. Then, we can write 
\[\ga_n:=\set{(a_i, a_i\sqrt 3) : 1\leq i\leq n},\]
where $0<a_1<a_2<a_3<\cdots<a_n<1$. In fact, by Note~\ref{note11}, we have  $0\leq a_1<a_2<a_3<\cdots<a_n\leq \frac 12$.
The quantization error $V_n$ is given by 
\begin{align*}
V_n&=\frac 12\Big(\int_0^{2(a_1+a_2)}\rho(x, (a_1, a_1\sqrt 3)) \, dx+\sum_{i=2}^{n-1}\int_{2(a_{i-1}+a_i)}^{2(a_{i}+a_{i+1})}\rho(x, (a_i, a_i\sqrt 3)) \, dx\\
&\qquad   + \int_{2(a_{n-1}+a_n)}^2\rho(x, (a_n, a_n\sqrt 3)) \,dx\Big).
\end{align*}
Since $V_n$ gives the optimal error and is differentiable with respect to $a_i$, we have $\frac{\pa}{\pa a_i} V_n=0$ for $1\leq i\leq n$.
Solving the equations $\frac{\pa}{\pa a_i} V_n=0$ for $1\leq i\leq n-1$, we have 
\[2a_1=a_2-a_1=a_3-a_2=\cdots=a_n-a_{n-1}=u  \te{ implying } a_i=\frac {(2i-1)u}{2} \te{ for } 1\leq i\leq n,\]
 where $u$ is a constant depending on $n$ such that $0<u<1$. 
 Now, putting the values of $a_i$ in term of $u$, we have 
 \begin{align*}
 \int_0^{2(a_1+a_2)}\rho(x, (a_1, a_1\sqrt 3)) \, dx&=\frac{52 u^3}{3},\\
 \sum_{i=2}^{n-1}\int_{2(a_{i-1}+a_i)}^{2(a_{i}+a_{i+1})}\rho(x, (a_i, a_i\sqrt 3)) \, dx&=\frac{4}{3} \Big(12 n^3 u^3-36 n^2 u^3+37 n u^3-26 u^3\Big),\\
 \int_{2(a_{n-1}+a_n)}^2\rho(x, (a_n, a_n\sqrt 3)) \,dx&=-\frac{1}{3} 4 \left(16 n^3-42 n^2+39 n-13\right) u^3\\&
 \hspace{ 0.5 in} +2 (1-2 n)^2 u^2+(2-4 n) u+\frac{8}{3}.
\end{align*}
 Now, substituting all the corresponding values in the expression for $V_n$, and then upon simplification,  we have 
\begin{align*}
V_n&=-\frac{1}{3} 4 n \left(2 n^2-3 n+1\right) u^3+(1-2 n)^2 u^2-2 n u+u+\frac{4}{3},
\end{align*}
the minimum value of which occurs when $u=\frac{1}{2 n}$. 
Hence, putting the values of $u$ in the last expression for $V_n$, we obtain
\[V_n=\frac{1}{12 n^2}+1\geq 1>V_3 \te{ for all } n\geq 3,\]
which leads to a contradiction. This contradiction arises because $V_n$ is a decreasing sequence, and by Proposition~\ref{propD2}, we know $V_3=\frac{1}{108} (-637) \left(28 \sqrt{13}-101\right)<1$. 
Thus, the proof of the lemma is complete. 
\end{proof}

\begin{lemma} \label{lemmaD2}
For $n\geq 4$, let $\ga_n$ be an optimal set of $n$-points for $P$. Then, $\ga_n$ must contain at least two elements from each of $S_1$ and $S_2$. 
\end{lemma} 
\begin{proof}
 By Remark~\ref{remD1}, the lemma is true for $n=4$. Let us now prove the lemma for $n\geq 5$. Let $\ga_n$ be an optimal set of $n$-points for $n\geq 5$ with the corresponding $n$th constrained quantization error $V_n$. By Lemma~\ref{lemmaD1}, $\ga_n$ contains at least one element from each of $S_1$ and $S_2$. We need to show that $\ga_n$ must contain at least two elements from each of $S_1$ and $S_2$. For the sake of contradiction, without any loss of generality, we can assume that $\ga_n$ contains exactly one element from  $S_2$. Then, $\ga_n$ contains $\ell:=n-1$ elements from $S_1$. 
We can write 
\[\ga_n:=\set{(a_i, a_i\sqrt 3) : 1\leq i\leq \ell}\uu \set{(b_1, -(b_1-2)\sqrt 3},\]
where $0<a_1<a_2<a_3<\cdots<a_{\ell}<1$ and $1<b_1<2$. In fact, by Note~\ref{note11}, we have  $0\leq a_1<a_2<a_3<\cdots<a_{\ell}\leq \frac 12$.
Let $(d, 0)$ be the point where the boundary of the Voronoi regions of $(a_\ell, a_\ell \sqrt 3)$ and $(b_1, -(b_1-2) \sqrt 3)$ intersects the support of $P$. Then, solving the canonical equation 
\[\rho((a_\ell, a_\ell\sqrt 3), (d, 0))-\rho((b_1, -(b_1-2)\sqrt 3), (d, 0))=0,\]
we have 
\begin{equation} \label{eq31} d=\frac{4 b_1^2-12 b_1-(1-2 l)^2 u^2+12}{2 b_1-2 l u+u}.
\end{equation} The quantization error $V_n$ is given by 
\begin{align*}
V_n&=\frac 12\Big(\int_0^{2(a_1+a_2)}\rho(x, (a_1, a_1\sqrt 3)) \, dx+\sum_{i=2}^{\ell-1}\int_{2(a_{i-1}+a_i)}^{2(a_{i}+a_{i+1})}\rho(x, (a_i, a_i\sqrt 3)) \, dx\\
&\qquad   + \int_{2(a_{\ell-1}+a_\ell)}^d\rho(x, (a_\ell, a_\ell\sqrt 3)) \,dx+\int_d^2\rho(x, (b_1, -(b_1-2)\sqrt 3))\,dx \Big).
\end{align*}
Since $V_n$ gives the optimal error and is differentiable with respect to $a_i$ and $b_j$, we have $\frac{\pa}{\pa a_i} V_n=0$ and $\frac{\pa}{\pa b_j} V_n=0$ for $1\leq i\leq \ell$, and $j=1$.
Solving the equations $\frac{\pa}{\pa a_i} V_n=0$ for $1\leq i\leq \ell-1$, we have 
\[2a_1=a_2-a_1=a_3-a_2=\cdots=a_\ell-a_{\ell-1}=u  \te{ implying } a_i=\frac {(2i-1)u}{2} \te{ for } 1\leq i\leq \ell,\]
where $u$ is a constant depending on $\ell$ such that $0<u<1$. 
 Now, putting the values of $a_i$, we have 
 \begin{align*}
 \int_0^{2(a_1+a_2)}\rho(x, (a_1, a_1\sqrt 3)) \, dx&=\frac{52 u^3}{3},\\
 \sum_{i=2}^{\ell-1}\int_{2(a_{i-1}+a_i)}^{2(a_{i}+a_{i+1})}\rho(x, (a_i, a_i\sqrt 3)) \, dx&=\frac{4}{3} \Big(12 \ell^3 u^3-36 \ell^2 u^3+37 \ell u^3-26 u^3\Big).
\end{align*}
Moreover, putting the values of $a_{\ell-1},\, a_\ell$ and $d$ in terms of $\ell$, $u$ and $b_1$, we have 
\begin{align*}
 &\int_{2(a_{\ell-1}+a_\ell)}^d\rho(x, (a_\ell, a_\ell\sqrt 3)) \,dx+\int_d^2\rho(x, (b_1, -(b_1-2)\sqrt 3))\,dx\\
 &=b_1 (1-2 \ell)^2 u^2-2 ((b_1-6) b_1+6) (2 \ell-1) u-\frac{72 (b_1-1)^2}{2 b_1-2 \ell u+u}-4 b_1 ((b_1-8) b_1+13)\\
 &\qquad +\frac{1}{6} (101-2 \ell (2 \ell (26 \ell-75)+147)) u^3+\frac{80}{3}.
\end{align*}
Now, putting all the corresponding values in the expression for $V_n$, we have 
\begin{align*}
V_n&=\frac{1}{2}\Big(\frac{52 u^3}{3} +\frac{4}{3} \left(12 \ell^3 u^3-36 \ell^2 u^3+37 \ell u^3-26 u^3\right)+b_1 (1-2 \ell)^2 u^2-2 ((b_1-6) b_1+6) (2 \ell-1) u\\
&\qquad -\frac{72 (b_1-1)^2}{2 b_1-2 \ell u+u}-4 b_1 ((b_1-8) b_1+13)+\frac{1}{6} (101-2 \ell (2 \ell (26 \ell-75)+147)) u^3+\frac{80}{3}\Big).
\end{align*}
Now, notice that $V_n$ is a function of $\ell, \,u $ and $b_1$. $V_n$ being optimal we have 
\[
\frac{\pa}{\pa u} V_n=0  \te{ and } \frac{\pa}{\pa b_1} V_n=0.\] 
Solving the above two equations, we deduce that
\[u=-\frac{\sqrt{156 \ell^2+13}-13 \ell}{2 \left(\ell^2-1\right)} \te{ and } b_1=-\frac{-20 \ell^2+\sqrt{156 \ell^2+13} \ell+7}{4 \left(\ell^2-1\right)}.\]
Hence, putting the values of $u$ and $b_1$, in the last expression for $V_n$, we obtain
\begin{equation} \label{eqD3} V_n=\frac{1}{\frac{24 l}{\sqrt{156 l^2+13}}+\frac{1}{-12 l^2-1}+\frac{25}{13}}=\frac{1}{\frac{24 (n-1)}{\sqrt{156 (n-1)^2+13}}+\frac{1}{-12 (n-1)^2-1}+\frac{25}{13}}.
\end{equation} 
Recall that the $n$th constrained quantization error $V_n$ is a decreasing sequence of real numbers, i.e., $V_1\geq V_2\geq V_3\geq V_4\geq V_n$ for all $n\geq 5$. Also, recall that by Remark~\ref{remD1}, we have $V_4=\frac{49}{192}$. But, by \eqref{eqD3}, we see that for all $n\geq 5$, 
\[V_n\geq \lim_{n\to\infty} V_n=\frac{1}{4 \sqrt{\frac{3}{13}}+\frac{25}{13}}>V_4,\]
which yields a contradiction. Thus, we conclude that for all $n\geq 4$, the set $\ga_n$ must contain at least two elements from each of $S_1$ and $S_2$, which is the lemma. 
\end{proof}

\begin{lemma} \label{lemmaD3}
For $n\geq 6$, let $\ga_n$ be an optimal set of $n$-points for $P$. Then, $\ga_n$ must contain at least three elements from each of $S_1$ and $S_2$. 
\end{lemma}
\begin{proof}
For $n\geq 6$, let $\ga_n$ be an optimal set of $n$-points for $P$. Let $\te{card}(\ga_n\ii S_1)=\ell$ and $\te{card}(\ga_n\ii S_2)=m$. 
By Lemma~\ref{lemmaD2}, $\ell, m\geq 2$. We need to prove that $\ell, m\geq 3$. For the sake of contradiction, without any loss of generality, we can assume that $\ga_n$ contains exactly two elements from  $S_2$, i.e., $m=2$. Then, $\ga_n$ contains $\ell:=n-2$ elements from $S_1$. 
We can write 
\[\ga_n:=\set{(a_i, a_i\sqrt 3) : 1\leq i\leq \ell}\uu \set{(b_2, -(b_2-2)\sqrt 3), (b_1, -(b_1-2)\sqrt 3)},\]
where $0<a_1<a_2<a_3<\cdots<a_{\ell}<1$ and $1<b_2<b_1<2$. In fact, by Note~\ref{note11}, we have  $0\leq a_1<a_2<a_3<\cdots<a_{\ell}\leq \frac 12$ and $\frac 32\leq b_2<b_1\leq 2$.
Let $(d, 0)$ be the point where the boundary of the Voronoi regions of $(a_\ell, a_\ell \sqrt 3)$ and $(b_2, -(b_2-2) \sqrt 3)$ intersects the support of $P$. Then, solving the canonical equation 
\[\rho((a_\ell, a_\ell\sqrt 3), (d, 0))-\rho((b_2, -(b_2-2)\sqrt 3), (d, 0))=0,\]
we have 
\begin{equation} \label{eq31} d=\frac{2 \left(a_l^2-b_2^2+3 b_2-3\right)}{a_l-b_2}.
\end{equation} The quantization error $V_n$ is given by 
\begin{align*}
V_n&=\frac 12\Big(\int_0^{2(a_1+a_2)}\rho(x, (a_1, a_1\sqrt 3)) \, dx+\sum_{i=2}^{\ell-1}\int_{2(a_{i-1}+a_i)}^{2(a_{i}+a_{i+1})}\rho(x, (a_i, a_i\sqrt 3)) \, dx\\
&\qquad   + \int_{2(a_{\ell-1}+a_\ell)}^d\rho(x, (a_\ell, a_\ell\sqrt 3)) \,dx+\int_d^{2(b_1+b_2)-6} \rho(x, (b_1, -(b_1-2)\sqrt 3))\,dx \\
&\qquad +\int_{2(b_1+b_2)-6}^2 \rho(x, (b_2, -(b_2-2)\sqrt 3))\,dx\Big).
\end{align*}
Since $V_n$ gives the optimal error and is differentiable with respect to $a_i$ and $b_j$, we have $\frac{\pa}{\pa a_i} V_n=0$ and $\frac{\pa}{\pa b_j} V_n=0$ for $1\leq i\leq \ell$, and $j=1$.
Solving the equations $\frac{\pa}{\pa a_i} V_n=0$ for $1\leq i\leq \ell-1$, we have 
\[2a_1=a_2-a_1=a_3-a_2=\cdots=a_\ell-a_{\ell-1}=u  \te{ implying } a_i=\frac {(2i-1)u}{2} \te{ for } 1\leq i\leq \ell,\]
where $u$ is a constant depending on $\ell$ such that $0<u<1$.  Solving the equations $\frac{\pa}{\pa b_1} V_n=0$, we have $b_2=3b_1-4$. 
 Now, putting the values of $a_i$, we have 
 \begin{align*}
 \int_0^{2(a_1+a_2)}\rho(x, (a_1, a_1\sqrt 3)) \, dx&=\frac{52 u^3}{3},\\
 \sum_{i=2}^{\ell-1}\int_{2(a_{i-1}+a_i)}^{2(a_{i}+a_{i+1})}\rho(x, (a_i, a_i\sqrt 3)) \, dx&=\frac{4}{3} \Big(12 \ell^3 u^3-36 \ell^2 u^3+37 \ell u^3-26 u^3\Big).
\end{align*}
Moreover, putting the values of $a_{\ell-1},\, a_\ell$ and $d$ in terms of $\ell$, $u$ and $b_1$, we have 
\begin{align*}
 &\int_{2(a_{\ell-1}+a_\ell)}^d\rho(x, (a_\ell, a_\ell\sqrt 3)) \,dx+\int_d^{2(b_1+b_2)-6} \rho(x, (b_1, -(b_1-2)\sqrt 3))\,dx \\
 &\qquad +\int_{2(b_1+b_2)-6}^2 \rho(x, (b_2, -(b_2-2)\sqrt 3))\,dx\\
 &=b_1\Big(-6 b_1 ((6 l-3) u+8)+20 b_1^2+3 (2 l-1) u ((2 l-1) u+28)-72\Big)+\frac{18 (-2 l u+u+2)^2}{-6 b_1+(2 l-1) u+8}\\
 &\qquad +\Big(-\frac{52 l^3}{3}+50 l^2-49 l+\frac{101}{6}\Big) u^3-4 (1-2 l)^2 u^2+110 (1-2 l) u+\frac{584}{3}.
\end{align*}
Now, putting all the corresponding values in the expression for $V_n$, we have 
\begin{align*}
V_n&=\frac{1}{2}\Big(\frac{52 u^3}{3} +\frac{4}{3} \Big(12 \ell^3 u^3-36 \ell^2 u^3+37 \ell u^3-26 u^3\Big)+b_1 \Big(-6 b_1 ((6 l-3) u+8)+20 b_1^2\\
&+3 (2 l-1) u ((2 l-1) u+28)-72\Big)+\frac{18 (-2 l u+u+2)^2}{-6 b_1+(2 l-1) u+8}\\
 &\qquad +\Big(-\frac{52 l^3}{3}+50 l^2-49 l+\frac{101}{6}\Big) u^3-4 (1-2 l)^2 u^2+110 (1-2 l) u+\frac{584}{3}\Big).
\end{align*}
Now, notice that $V_n$ is a function of $\ell, \,u $ and $b_1$. $V_n$ being optimal we have 
\[
\frac{\pa}{\pa u} V_n=0  \te{ and } \frac{\pa}{\pa b_1} V_n=0.\] 
Solving the above two equations, we deduce that
\[u= \frac{1}{\frac{4}{7} \sqrt{12 l^2+1}+2 l} \te{ and } b_1= \frac{-32 l^2+7 \sqrt{12 l^2+1} l+30}{16-4 l^2}.\]
Hence, putting the values of $u$ and $b_1$, in the last expression for $V_n$, we obtain
\begin{equation} \label{eqD3} V_n=\frac{1}{\frac{48 l}{7 \sqrt{12 l^2+1}}+\frac{1}{-12 l^2-1}+\frac{97}{49}}=\frac{1}{\frac{48 (n-2)}{7 \sqrt{12 (n-2)^2+1}}+\frac{1}{-12 (n-2)^2-1}+\frac{97}{49}}.
\end{equation} 
Now, putting $n=6$ in the above expression for $V_n$, i.e., when $\ell=4$ and $m=2$, we see that $V_6=0.253244$, which is larger than the value of $V_6$ when $\ell=m=3$ (see Example~\ref{exam1}). 
Again, $V_n$ being a decreasing sequence, we must have $V_7\geq V_8\geq V_9\geq \cdots$, i.e., $V_7\geq \mathop{\lim}\limits_{n\to \infty} V_n=0.252584$, which is larger than the value of $V_7$ when it is obtained if $\ell=4$ and $m=3$. (see Example~\ref{exam2}).  Thus, we conclude that for all $n\geq 6$, the expression for $V_n$ given by \eqref{eqD3}, gives a contradiction. Hence, we can conclude that for $n\geq 6$, the optimal set $\ga_n$ must contain at least three elements from each of $S_1$ and $S_2$, which is the lemma. 
\end{proof}

Let us now state and prove the following theorem, which gives all the optimal sets of $n$-points and the $n$th constrained quantization error for all $n\geq 6$. 

\begin{theorem} \label{theo1}
Let $n\in \D N$ with $n\geq 6$. Let $\ga_n$ be an optimal set of $n$-points such that $\te{card}(\ga_n\ii S_1)=\ell$ and $\te{card}(\ga_n\ii S_2)=m$ with $\ell, m\geq 3$. Then, 
\[\ga_n=\Big\{(a_i, a_i\sqrt 3) : 1\leq i\leq \ell\Big\}\UU \Big\{(b_j, -(b_j-2)\sqrt 3) :  1\leq j\leq m\Big\},\]
where $a_i=\frac {(2i-1)u}{2}$ and $b_j=2-\frac {(2j-1)v}{2}$, where $u$, $v$ are given by 
 \[u=\frac{1}{2 \left(\frac{\left(12 \ell^2+1\right) m}{\sqrt{\left(12 \ell^2+1\right) \left(12 m^2+1\right)}}+\ell\right)} \te{ and } v=\frac{1}{2 \left(\frac{\ell \left(12 m^2+1\right)}{\sqrt{\left(12 \ell^2+1\right) \left(12 m^2+1\right)}}+m\right)},\] and the constrained quantization error for $n$-points is given by 
\[V_n=\frac{\left(12 \ell^2+1\right) \left(12 m^2+1\right)}{12 \left(24 \ell^2 m^2+2 \ell m \sqrt{\left(12 \ell^2+1\right) \left(12 m^2+1\right)}+\ell^2+m^2\right)}.\]
 \end{theorem}

\begin{proof}
Let $\ga_n$ be an optimal set of $n$-points for $n\geq 6$ such that $\ga_n$ contains $\ell$ elements from $S_1$, and $m$ elements from $S_2$ with $\ell, m\geq 3$. Let
\[\ga_n\ii S_1:=\set{(a_i, a_i\sqrt 3) : 1\leq i\leq \ell} \te{ and }\ga_n\ii S_2:=\set{(b_j, -(b_j-2)\sqrt 3) : 1\leq j\leq m},\]
where $a_1<a_2<a_3<\cdots<a_\ell$ and $b_m<b_{m-1}<b_{m-2}<\cdots<b_1$. In fact, by Note~\ref{note11}, we have  $0\leq a_1<a_2<a_3<\cdots<a_{\ell}\leq \frac 12$ and $\frac 32\leq b_m<b_{m-1}<b_{m-2}<\cdots<b_1\leq 2$. 
Let $(d, 0)$ be the point where the boundary of the Voronoi regions of $(a_\ell, a_\ell \sqrt 3)$ and $(b_m, -(b_m-2) \sqrt 3)$ intersects the support of $P$. Then, solving the canonical equation 
\[\rho((a_\ell, a_\ell\sqrt 3), (d, 0))-\rho((b_m, -(b_m-2)\sqrt 3), (d, 0))=0,\]
we have 
\begin{equation} \label{eq31} d=\frac{2 (a_\ell^2-b_m^2+3 b_m-3)}{a_\ell-b_m}.
\end{equation} The quantization error $V_n$ is given by 
\begin{align*}
V_n&=\frac 12\Big(\int_0^{2(a_1+a_2)}\rho(x, (a_1, a_1\sqrt 3)) \, dx+\sum_{i=2}^{\ell-1}\int_{2(a_{i-1}+a_i)}^{2(a_{i}+a_{i+1})}\rho(x, (a_i, a_i\sqrt 3)) \, dx\\
&\qquad   + \int_{2(a_{\ell-1}+a_\ell)}^d\rho(x, (a_\ell, a_\ell\sqrt 3)) \,dx+\int_d^{2(b_m+b_{m-1})-6}\rho(x, (b_m, -(b_m-2)\sqrt 3))\,dx \\
&\qquad +\sum_{j=2}^{m-1}\int_{2(b_{j+1}+b_j)-6}^{2(b_j+b_{j-1})-6}\rho(x, (b_j, -(b_j-2)\sqrt 3)) \, dx+\int_{2(b_2+b_1)-6}^2 \rho(x, (b_1, -(b_1-2)\sqrt 3))\,dx\Big).
\end{align*}
Since $V_n$ gives the optimal error and is differentiable with respect to $a_i$ and $b_j$, we have $\frac{\pa}{\pa a_i} V_n=0$ and $\frac{\pa}{\pa b_j} V_n=0$ for $1\leq i\leq \ell$, and $1\leq j\leq m$.
Solving the equations $\frac{\pa}{\pa a_i} V_n=0$ for $1\leq i\leq \ell-1$, we have 
\[2a_1=a_2-a_1=a_3-a_2=\cdots=a_\ell-a_{\ell-1}=u  \te{ implying } a_i=\frac {(2i-1)u}{2} \te{ for } 1\leq i\leq \ell,\]
where $u$ is a constant depending on $\ell$ such that $0<u<1$. Similarly, solving the equations $\frac{\pa}{\pa b_j} V_n=0$ for $1\leq j\leq m-1$, we have 
\[4-2b_1=b_1-b_2=b_2-b_3=\cdots=b_{m-1}-b_m=v \te{ implying } b_j=2-\frac {(2j-1)v}{2} \te{ for } 1\leq j\leq m,\]
 where $v$ is a constant depending on $m$ such that $0<v<1$. Putting $a_\ell=\frac{1}{2} (2 \ell-1) u$ and $b_m= 2-\frac{1}{2} (2 m-1) v$ in \eqref{eq31}, we have 
 \[d=\frac{-4 \ell^2 u^2+4 \ell u^2+4 m^2 v^2-4 m v^2-4 m v-u^2+v^2+2 v+4}{-2 \ell u-2 m v+u+v+4}.\]
 Now, putting the values of $a_i$ and $b_j$, we have 
 \begin{align*}
 \int_0^{2(a_1+a_2)}\rho(x, (a_1, a_1\sqrt 3)) \, dx&=\frac{52 u^3}{3},\\
 \int_{2(b_2+b_1)-6}^2 \rho(x, (b_1, -(b_1-2)\sqrt 3))\,dx&=\frac{52 v^3}{3},\\
 \sum_{i=2}^{\ell-1}\int_{2(a_{i-1}+a_i)}^{2(a_{i}+a_{i+1})}\rho(x, (a_i, a_i\sqrt 3)) \, dx&=\frac{4}{3} \Big(12 \ell^3 u^3-36 \ell^2 u^3+37 \ell u^3-26 u^3\Big),\\
 \sum_{j=2}^{m-1}\int_{2(b_{j+1}+b_j)-6}^{2(b_j+b_{j-1})-6}\rho(x, (b_j, -(b_j-2)\sqrt 3)) \, dx&=\frac{4}{3} \Big(12 m^3 v^3-36 m^2 v^3+37 m v^3-26 v^3\Big).
\end{align*}
Moreover, putting the values of $a_{\ell-1},\, a_\ell,\, b_m\, b_{m-1}$, and $d$ in terms of $\ell, m, u$ and $v$, we have 
\begin{align*}
 &\int_{2(a_{\ell-1}+a_\ell)}^d\rho(x, (a_\ell, a_\ell\sqrt 3)) \,dx+\int_d^{2(b_m+b_{m-1})-6}\rho(x, (b_m, -(b_m-2)\sqrt 3))\,dx\\
 &=\frac{1}{6 ((2 \ell-1) u+(2 m-1) v-4)}\Big((8 \ell (\ell (-26 \ell^2+88 \ell-111)+62)-101) u^4-16\\
 &+32 (m-1) (2 m (8 m-13)+13) v^3-24 (1-2 m)^2 v^2+16 (2 m-1) v\\
 &+(8 m (m (-26 m^2+88 m-111)+62)-101) v^4-8 (\ell-1) (2 \ell (8 \ell-13)+13) u^3 ((2 m-1) v-4)\\
 &-6 (u-2 \ell u)^2 ((2 m-1) v ((2 m-1) v-2)+4)-4 (2 \ell-1) u (2 (m-1) v-1) (v ((2 m (8 m-13)+13) v\\
 &+2 m-5)+4)\Big).
\end{align*}
Now, putting all the corresponding values in the expression for $V_n$, we have 
\begin{align*}
V_n&=-\frac{1}{12 ((2 \ell-1) u+(2 m-1) v-4)}\Big(16 \ell (2 \ell^2-3 \ell+1) u^3 ((2 m-1) v-4)+(1-2 \ell)^2 (4 \ell^2-4 \ell-3) u^4\\
&+4 (2 \ell-1) u (4 m (2 m^2-3 m+1) v^3-3 (1-2 m)^2 v^2+(6 m-3) v-4)\\
&+6 (u-2 \ell u)^2 ((1-2 m)^2 v^2+(2-4 m) v+4)+(1-2 m)^2 (4 m^2-4 m-3) v^4\\
&-64 m (2 m^2-3 m+1) v^3+24 (1-2 m)^2 v^2-16 (2 m-1) v+16\Big).
\end{align*}
Notice that $V_n$ is a function of $\ell, \, m, \,u$ and $v$. $V_n$ being optimal we have 
\[
\frac{\pa}{\pa u} V_n=0  \te{ and } \frac{\pa}{\pa v} V_n=0.\] 
Solving the above two equations, we deduce that
\[u=\frac{1}{2 \left(\frac{\left(12 \ell^2+1\right) m}{\sqrt{\left(12 \ell^2+1\right) \left(12 m^2+1\right)}}+\ell\right)} \te{ and } v=\frac{1}{2 \left(\frac{\ell \left(12 m^2+1\right)}{\sqrt{\left(12 \ell^2+1\right) \left(12 m^2+1\right)}}+m\right)}.\]
After substituting the values of $u$ and $v$ in $V_n$, and then upon simplification, we have 
\[V_n=\frac{\left(12 \ell^2+1\right) \left(12 m^2+1\right)}{12 \left(24 \ell^2 m^2+2 \ell m \sqrt{\left(12 \ell^2+1\right) \left(12 m^2+1\right)}+\ell^2+m^2\right)}.\]
Thus, the proof of the theorem is complete. 
\end{proof} 

\begin{remark}\label{remD11}
From the expression of $V_n$ given by Theorem~\ref{theo1}, we see that $V_n$ is symmetric in $\ell$, $m$. Since $V_n$ is optimal and symmetric in $\ell, \, m$, we can deduce the following: If $n$ is of the form $n=2k$ for some $k\in \D N$, then $\ell=m=k$; on the other hand,  if $n$ is of the form $n=2k+1$ for some $k\in \D N$, then either $\ell=k+1$ and $m=k$, or $\ell=k$ and $m=k+1$. 
\end{remark} 
By Theorem~\ref{theo1} and Remark~\ref{remD11}, we deduce the following corollary. The corollary is needed to calculate the constrained quantization dimension and constrained quantization coefficient given in Section~\ref{sec4}. 
 
\begin{corollary}\label{corD1}
Let $n\geq 6$ be an even positive integer, i.e., $n=2k$ for some positive integer $k\geq 3$. Then, the set $\ga_n$ given by Theorem~\ref{theo1} forms an optimal set of $n$-points if $\ell=m=k$, and the corresponding constrained quantization error is given by 
\[V_n=\frac{1}{48} \left(12+\frac{1}{k^2}\right).\]
\end{corollary} 
We now give the following two examples. 
\begin{exam} \label{exam1} \tit{Optimal set of six-points:} For $n=6$, we have $\ell=m=3$, and hence, $u=v=\frac 1{12}$ implying 
\begin{align*}
\ga_6=\left\{(a_i, a_i\sqrt 3) : 1\leq i\leq 3\right\}\UU \left\{(b_j, -(b_j-2)\sqrt 3) : 1\leq j\leq 3\right\},
\end{align*}
where $a_i=\frac {(2i-1)u}{2}$ and $b_j=2-\frac {(2j-1)v}{2}$ for $1\leq i, j\leq 3$. 
Thus, we have 
\begin{align*}
\ga_6=\Big\{(\frac{1}{24},\frac{1}{8 \sqrt{3}}), (\frac{1}{8},\frac{\sqrt{3}}{8}),(\frac{5}{24},\frac{5}{8 \sqrt{3}}), (\frac{43}{24},\frac{5}{8 \sqrt{3}}), (\frac{15}{8},\frac{\sqrt{3}}{8}), (\frac{47}{24},\frac{1}{8 \sqrt{3}})\Big\}
\end{align*}
with the constrained quantization error $V_6=\frac{109}{432}=0.252315$ (see Figure~\ref{Fig1}). 
\end{exam} 

\begin{exam}\label{exam2} \tit{Optimal set of seven-points:} For $n=7$, we have either $\ell=4$ and $m=3$, or $\ell=3$ and $m=4$. Let us take $\ell=4$ and $m=3$, and then 
\[u=\frac{1}{2 \left(3 \sqrt{\frac{193}{109}}+4\right)} \te{ and } v=\frac{1}{8 \sqrt{\frac{109}{193}}+6}\] implying 
\begin{align*}
\ga_7=\left\{(a_i, a_i\sqrt 3) : 1\leq i\leq 4\right\}\UU \left\{(b_j, -(b_j-2)\sqrt 3) : 1\leq j\leq 3\right\},
\end{align*}
where $a_i=\frac {(2i-1)u}{2}$ and $b_j=2-\frac {(2j-1)v}{2}$ for $1\leq i\leq 4$ and $1\leq j\leq 3$. Thus, we have 
\begin{align*}
\ga_7&=\Big\{(0.0312814,0.054181),(0.0938443,0.162543),(0.156407,0.270905), (0.21897,0.379267),\\
&\qquad (1.79188,0.360481),(1.87513,0.216289),(1.95838,0.0720962)\Big\}
\end{align*}
with the constrained quantization error $V_7=\frac{21037}{288 \sqrt{21037}+41772}=0.251808$ (see Figure~\ref{Fig1}). 
\end{exam}
 
 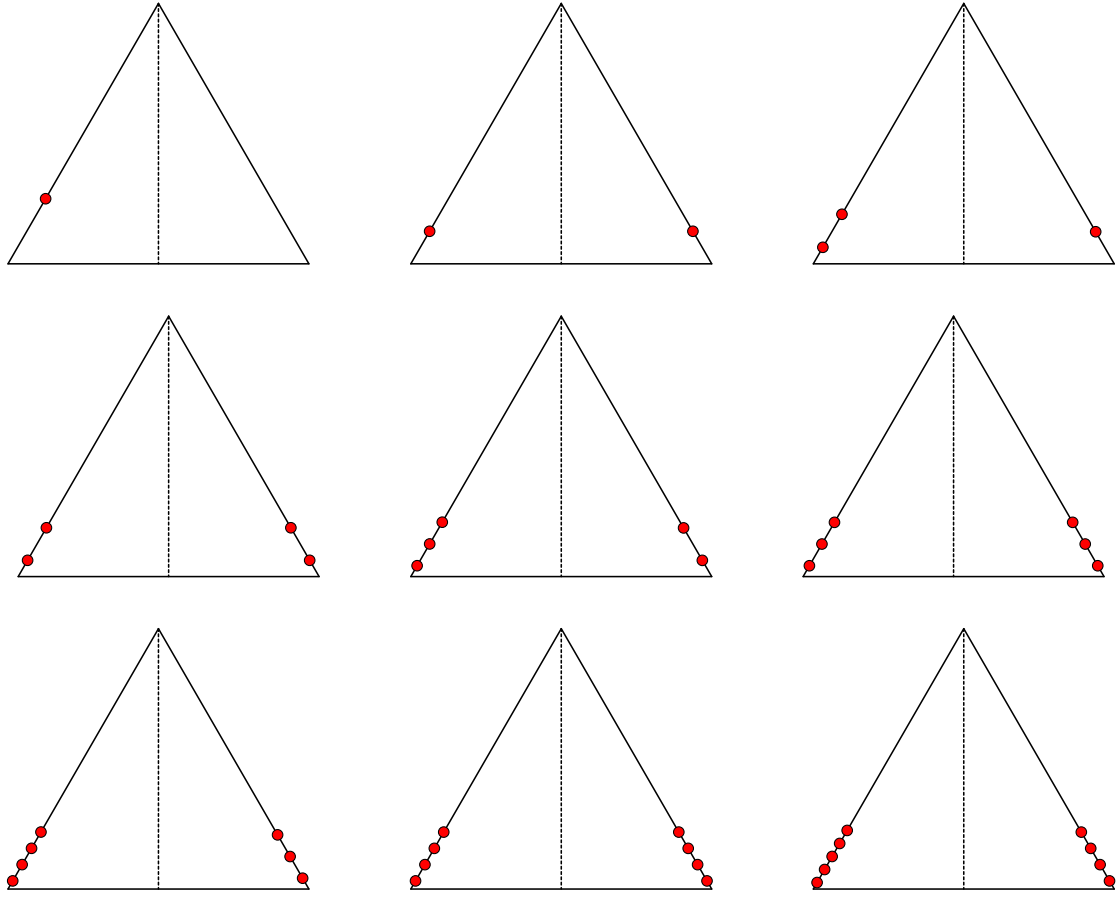
\begin{figure}
  \begin{tikzpicture}[line cap=round,line join=round,>=triangle 45, x=1.0cm,y=1.0cm]
\clip(-0.270613645021906,-0.38022889680252) rectangle (4.53303695897085,3.743191138497082);
\draw [line width=0.59 pt] (0.,0.)-- (4.,0.);
\draw [line width=0.59 pt] (0.,0.)-- (2.,3.4641016151377544);
\draw [line width=0.59 pt] (2.,3.4641016151377544)-- (4.,0.);
\draw [line width=0.59 pt,dash pattern=on 1pt off 1pt] (2.,3.4641)-- (2.,0.);
\begin{scriptsize}
\draw [fill=ffqqqq] (0.5,0.8660254037844386) circle (2.0pt);
\end{scriptsize}
\end{tikzpicture} \quad   
\begin{tikzpicture}[line cap=round,line join=round,>=triangle 45, x=1.0cm,y=1.0cm]
\clip(-0.270613645021906,-0.38022889680252) rectangle (4.53303695897085,3.743191138497082);
\draw [line width=0.59 pt] (0.,0.)-- (4.,0.);
\draw [line width=0.59 pt] (0.,0.)-- (2.,3.4641016151377544);
\draw [line width=0.59 pt] (2.,3.4641016151377544)-- (4.,0.);
\draw [line width=0.59 pt,dash pattern=on 1pt off 1pt] (2.,3.4641)-- (2.,0.);
\begin{scriptsize}
\draw [fill=ffqqqq] (0.25, 0.433013) circle (2.0pt);
\draw [fill=ffqqqq] (3.75, 0.433013) circle (2.0pt);
\end{scriptsize}
\end{tikzpicture}
\quad  \begin{tikzpicture}[line cap=round,line join=round,>=triangle 45, x=1.0cm,y=1.0cm]
\clip(-0.270613645021906,-0.38022889680252) rectangle (4.53303695897085,3.743191138497082);
\draw [line width=0.59 pt] (0.,0.)-- (4.,0.);
\draw [line width=0.59 pt] (0.,0.)-- (2.,3.4641016151377544);
\draw [line width=0.59 pt] (2.,3.4641016151377544)-- (4.,0.);
\draw [line width=0.59 pt,dash pattern=on 1pt off 1pt] (2.,3.4641)-- (2.,0.);
\begin{scriptsize}
\draw [fill=ffqqqq] (0.126857, 0.219723) circle (2.0pt);
\draw [fill=ffqqqq] (0.380571, 0.659168) circle (2.0pt);
\draw [fill=ffqqqq] (3.75371, 0.42658) circle (2.0pt);
\end{scriptsize}
\end{tikzpicture}
\quad
   \begin{tikzpicture}[line cap=round,line join=round,>=triangle 45, x=1.0cm,y=1.0cm]
\clip(-0.270613645021906,-0.38022889680252) rectangle (4.53303695897085,3.743191138497082);
\draw [line width=0.59 pt] (0.,0.)-- (4.,0.);
\draw [line width=0.59 pt] (0.,0.)-- (2.,3.4641016151377544);
\draw [line width=0.59 pt] (2.,3.4641016151377544)-- (4.,0.);
\draw [line width=0.59 pt,dash pattern=on 1pt off 1pt] (2.,3.4641)-- (2.,0.);
\begin{scriptsize}
\draw [fill=ffqqqq] (0.125, 0.216506) circle (2.0pt);
\draw [fill=ffqqqq] (0.375, 0.649519) circle (2.0pt);
\draw [fill=ffqqqq] (3.625, 0.649519) circle (2.0pt);
\draw [fill=ffqqqq] (3.875, 0.216506) circle (2.0pt);
\end{scriptsize}
\end{tikzpicture}\quad
\begin{tikzpicture}[line cap=round,line join=round,>=triangle 45, x=1.0cm,y=1.0cm]
\clip(-0.270613645021906,-0.38022889680252) rectangle (4.53303695897085,3.743191138497082);
\draw [line width=0.59 pt] (0.,0.)-- (4.,0.);
\draw [line width=0.59 pt] (0.,0.)-- (2.,3.4641016151377544);
\draw [line width=0.59 pt] (2.,3.4641016151377544)-- (4.,0.);
\draw [line width=0.59 pt,dash pattern=on 1pt off 1pt] (2.,3.4641)-- (2.,0.);
\begin{scriptsize}
\draw [fill=ffqqqq] (0.0835709, 0.144749) circle (2.0pt);
\draw [fill=ffqqqq] (0.250713, 0.434247) circle (2.0pt);
\draw [fill=ffqqqq] (0.417854, 0.723745) circle (2.0pt);
\draw [fill=ffqqqq] (3.62607, 0.647668) circle (2.0pt);
\draw [fill=ffqqqq] (3.87536, 0.215889) circle (2.0pt);
\end{scriptsize}
\end{tikzpicture}\quad
\begin{tikzpicture}[line cap=round,line join=round,>=triangle 45, x=1.0cm,y=1.0cm]
\clip(-0.270613645021906,-0.38022889680252) rectangle (4.53303695897085,3.743191138497082);
\draw [line width=0.59 pt] (0.,0.)-- (4.,0.);
\draw [line width=0.59 pt] (0.,0.)-- (2.,3.4641016151377544);
\draw [line width=0.59 pt] (2.,3.4641016151377544)-- (4.,0.);
\draw [line width=0.59 pt,dash pattern=on 1pt off 1pt] (2.,3.4641)-- (2.,0.);
\begin{scriptsize}
\draw [fill=ffqqqq] (0.0833333, 0.144338) circle (2.0pt);
\draw [fill=ffqqqq] (0.25, 0.433013) circle (2.0pt);
\draw [fill=ffqqqq] (0.416667, 0.721688) circle (2.0pt);
\draw [fill=ffqqqq] (3.58333, 0.721688) circle (2.0pt);
\draw [fill=ffqqqq] (3.75, 0.433013) circle (2.0pt);
\draw [fill=ffqqqq] (3.91667, 0.144338) circle (2.0pt);
\end{scriptsize}
\end{tikzpicture}
\quad
\begin{tikzpicture}[line cap=round,line join=round,>=triangle 45, x=1.0cm,y=1.0cm]
\clip(-0.270613645021906,-0.38022889680252) rectangle (4.53303695897085,3.743191138497082);
\draw [line width=0.59 pt] (0.,0.)-- (4.,0.);
\draw [line width=0.59 pt] (0.,0.)-- (2.,3.4641016151377544);
\draw [line width=0.59 pt] (2.,3.4641016151377544)-- (4.,0.);
\draw [line width=0.59 pt,dash pattern=on 1pt off 1pt] (2.,3.4641)-- (2.,0.);
\begin{scriptsize}
\draw [fill=ffqqqq] (0.0625628, 0.108362) circle (2.0pt);
\draw [fill=ffqqqq] (0.187689, 0.325086) circle (2.0pt);
\draw [fill=ffqqqq] (0.312814, 0.54181) circle (2.0pt);
\draw [fill=ffqqqq] (0.43794, 0.758534) circle (2.0pt);
\draw [fill=ffqqqq] (3.58375, 0.720962) circle (2.0pt);
\draw [fill=ffqqqq] (3.75025, 0.432577) circle (2.0pt);
\draw [fill=ffqqqq] (3.91675, 0.144192) circle (2.0pt);
\end{scriptsize}
\end{tikzpicture} \quad
\begin{tikzpicture}[line cap=round,line join=round,>=triangle 45, x=1.0cm,y=1.0cm]
\clip(-0.270613645021906,-0.38022889680252) rectangle (4.53303695897085,3.743191138497082);
\draw [line width=0.59 pt] (0.,0.)-- (4.,0.);
\draw [line width=0.59 pt] (0.,0.)-- (2.,3.4641016151377544);
\draw [line width=0.59 pt] (2.,3.4641016151377544)-- (4.,0.);
\draw [line width=0.59 pt,dash pattern=on 1pt off 1pt] (2.,3.4641)-- (2.,0.);
\begin{scriptsize}
\draw [fill=ffqqqq] (0.0625, 0.108253) circle (2.0pt);
\draw [fill=ffqqqq] (0.1875, 0.32476) circle (2.0pt);
\draw [fill=ffqqqq] (0.3125, 0.541266) circle (2.0pt);
\draw [fill=ffqqqq] (0.4375, 0.757772) circle (2.0pt);
\draw [fill=ffqqqq] (3.5625, 0.757772) circle (2.0pt);
\draw [fill=ffqqqq] (3.6875, 0.541266) circle (2.0pt);
\draw [fill=ffqqqq] (3.8125, 0.32476) circle (2.0pt);
\draw [fill=ffqqqq] (3.9375, 0.108253) circle (2.0pt);
\end{scriptsize}
\end{tikzpicture}
\quad
\begin{tikzpicture}[line cap=round,line join=round,>=triangle 45, x=1.0cm,y=1.0cm]
\clip(-0.270613645021906,-0.38022889680252) rectangle (4.53303695897085,3.743191138497082);
\draw [line width=0.59 pt] (0.,0.)-- (4.,0.);
\draw [line width=0.59 pt] (0.,0.)-- (2.,3.4641016151377544);
\draw [line width=0.59 pt] (2.,3.4641016151377544)-- (4.,0.);
\draw [line width=0.59 pt,dash pattern=on 1pt off 1pt] (2.,3.4641)-- (2.,0.);
\begin{scriptsize}
\draw [fill=ffqqqq] (0.0500233, 0.086643) circle (2.0pt);
\draw [fill=ffqqqq] (0.15007, 0.259929) circle (2.0pt);
\draw [fill=ffqqqq] (0.250117, 0.433215) circle (2.0pt);
\draw [fill=ffqqqq] (0.350163, 0.606501) circle (2.0pt);
\draw [fill=ffqqqq] (0.45021, 0.779787) circle (2.0pt);
\draw [fill=ffqqqq] (3.5627, 0.757419) circle (2.0pt);
\draw [fill=ffqqqq] (3.68765, 0.541013) circle (2.0pt);
\draw [fill=ffqqqq] (3.81259, 0.324608) circle (2.0pt);
\draw [fill=ffqqqq] (3.93753, 0.108203) circle (2.0pt);
\end{scriptsize}
\end{tikzpicture}
\caption{Optimal configuration of $n$-points for $1\leq n\leq 9$.} \label{Fig1}
\end{figure}

 \section{Constrained quantization dimension and constrained quantization coefficient} \label{sec4}
In this section, we show that the constrained quantization dimension $D(P)$ of the uniform distribution $P$ exists and equals one. We further show that the $D(P)$-dimensional constrained quantization coefficient exists as a finite positive number.

\begin{theorem}\label{theo3} 
The constrained quantization dimension $D(P)$ of the probability measure $P$ exists, and $D(P)=1$. 
\end{theorem}

\begin{proof}
For $n\in \D N$ with $n\geq 6$, let $\ell(n)$ be the unique natural number such that $2{\ell(n)}\leq n<2 ({\ell(n)+1})$. Then,
$V_{2(\ell(n)+1)}\leq V_n\leq V_{2{\ell(n)}}$. By Corollary~\ref{corD1}, we see that $V_{2({\ell(n)+1})}\to \frac 14$ and  $V_{2{\ell(n)}}\to \frac 14 $ as $n\to \infty$, and so $V_n\to \frac 14$ as $n\to \infty$,
i.e., $V_\infty=\frac 14$.
We can take $n$ large enough so that $ V_{2{\ell(n)}}-V_{\infty}<1$. Then,
\[0<-\log (V_{2{\ell(n)}}-V_{\infty})\leq -\log (V_n-V_\infty)\leq -\log (V_{2({\ell(n)+1})}-V_\infty)\]
yielding
\[\frac{2 \log 2\ell(n)}{-\log (V_{2({\ell(n)+1})}-V_\infty)}\leq \frac{2\log n}{-\log(V_n-V_\infty)}\leq \frac{2 \log 2(\ell(n)+1)}{-\log (V_{2{\ell(n)}}-V_\infty)}.\]
Notice that
\begin{align*}
\lim_{n\to \infty} \frac{2 \log 2\ell(n)}{-\log (V_{2({\ell(n)+1})}-V_\infty)}&=\lim_{n\to \infty} \frac{2 \log 2\ell(n)}{-\log \frac 1{48(\ell(n)+1)^2}}=1, \te{ and }\\
\lim_{n\to \infty} \frac{2 \log 2(\ell(n)+1)}{-\log (V_{2{\ell(n)}}-V_\infty)}&=\lim_{n\to \infty} \frac{2 \log 2(\ell(n)+1)}{-\log \frac 1{48(\ell(n))^2}}=1.
\end{align*}
Hence, $\lim_{n\to \infty}  \frac{2\log n}{-\log(V_n-V_\infty)}=1$, i.e., the constrained quantization dimension $D(P)$ of the probability measure $P$ exists and $D(P)=1$.
Thus, the proof of
the theorem is complete.
\end{proof}

\begin{theorem} \label{theo4} 
The $D(P)$-dimensional constrained quantization coefficient for $P$ exists as a finite positive number and equals $\frac 1{12}$.
\end{theorem}
\begin{proof}
For $n\in \D N$ with $n\geq 6$, let $\ell(n)$ be the unique natural number such that $2{\ell(n)}\leq n<2 ({\ell(n)+1})$. Then,
$V_{2(\ell(n)+1)}\leq V_n\leq V_{2{\ell(n)}}$, and $V_\infty=\lim_{n\to\infty} V_n=\frac 14$. Since
\begin{align*}
&\lim_{n\to \infty} n^2 (V_n-V_\infty)\geq \lim_{n\to \infty} (2\ell(n))^2 (V_{2(\ell(n)+1)}-V_\infty)=\lim_{n\to\infty}(2\ell(n))^2\frac 1{48(\ell(n)+1)^2}=\frac 1{12}, \te{ and } \\
&\lim_{n\to \infty} n^2 (V_n-V_\infty)\leq \lim_{n\to \infty} (2(\ell(n)+1))^2 (V_{2\ell(n)}-V_\infty)=\lim_{n\to\infty}(2(\ell(n)+1))^2\frac 1{48(\ell(n))^2}=\frac 1{12},
\end{align*}
by the squeeze theorem, we have $\lim_{n\to \infty} n^2 (V_n-V_\infty)=\frac 1{12}$, which is the theorem. 
\end{proof} 

\begin{remark} \label{rem1}
For the absolutely continuous probability measure, considered in this paper, we have obtained that the constrained quantization dimension is one which equals the dimension of the underlying space where the support of the probability measure is defined. This fact is not true, in general, in constrained quantization, for example, one can see \cite{PR1, PR3}. In this regard, we would like to mention that the unconstrained quantization dimension of an absolutely continuous probability measures always equals the dimension of the underlying space where the support of the probability measure is defined (see \cite{BW}).
\end{remark} 
 
%

\end{document}